\documentclass{amsart}
\usepackage{amsfonts}

\newtheorem{theorem}{Theorem}
\theoremstyle{plain}

\newtheorem{corollary}{Corollary}

\newtheorem{lemma}{Lemma}

\newtheorem{proposition}{Proposition}
\newtheorem{remark}{Remark}

\numberwithin{equation}{section}
\input{tcilatex}

\begin{document}
\title[Kesten--McKay]{On the generalized Kesten--McKay distributions}
\author{Pawe\l\ J. Szab\l owski}
\address{Warsaw University of Technology,\\
Department of Mathematics and Information Sciences}
\email{pawel.szablowski@gmail.com}
\date{June 2015}
\thanks{The author is deeply grateful to the unknown referee for careful
checking of all formulae and moreover pointing out numerous misprint and
misspellings found in the paper.}

\begin{abstract}
We examine the properties of distributions with the density of the form: $%
\frac{2A_{n}c^{n-2}\sqrt{c^{2}-x^{2}}}{\pi
\prod_{j=1}^{n}(c(1+a_{j}^{2})-2a_{j}x)},$ where $c,a_{1},\ldots ,a_{n}$ are
some parameters and $A_{n}$ a suitable constant. We find general forms of $%
A_{n}$, of $k-$th moment and of $k-$th polynomial orthogonal with respect to
such measures. We also calculate Cauchy transforms of these measures. We
indicate connections of such distributions with distributions and
polynomials forming the so called Askey--Wilson scheme. On the way we prove
several identities concerning rational symmetric functions. Finally, we
consider the case of parameters $a_{1},\ldots ,a_{n}$ forming conjugate
pairs and give some multivariate interpretations based on the obtained
distributions at least for the cases $n=2,4,6.$
\end{abstract}

\maketitle

\section{Introduction}

The purpose of this note is to analyze properties of the following family of
distributions having densities of the form:%
\begin{equation}
f_{KMKn}(x|c,a_{1},\ldots ,a_{n})\allowbreak =\allowbreak \frac{2A_{n}c^{n-2}%
\sqrt{c^{2}-x^{2}}}{\pi \prod_{j=1}^{n}(c(1+a_{j}^{2})-2a_{j}x)},
\label{genK}
\end{equation}%
defined for $n\geq 0,$ $\left\vert x\right\vert \leq c,$ with $c>0,$ $%
\left\vert a_{j}\right\vert <1,$ $j\allowbreak =\allowbreak 1,\ldots ,n.$
Here $A_{n}$ is a normalizing constant being the function of parameters $%
a_{1},\ldots ,a_{n}$. We will call this family generalized Kesten--McKay
distributions. The name is justified by the fact that the distribution with
the following density: 
\begin{equation}
f_{KMK2}(x|2/a,a,-a)\allowbreak =\allowbreak \frac{v\sqrt{4(v-1)-x^{2}}}{%
2\pi (v^{2}-x^{2})},  \label{r-K}
\end{equation}%
where $v\allowbreak =\allowbreak 1+1/a^{2}$ and $\left\vert a\right\vert $ $%
<1$ has been defined, described and what is more, derived in \cite{Kesten 59}
or \cite{McKay81}. Then the name Kesten-McKay distribution has been
attributed to this distribution in the literature that appeared after 1981.

Thus, it is justified to call the distribution defined by (\ref{genK}) a
generalized Kesten-McKay (GKM) distribution.

Note also that for $n\allowbreak =\allowbreak 0$ the distribution with the
density $f_{KMK0}(x|c)$ becomes Wigner or semicircle distribution with
parameter $c.$

It should be underlined that for $n=0,1,2$ distributions of this kind appear
not only in the context of random matrices, random graphs which is a typical
application of the Kesten--McKay distributions (see e.g.\cite{Oren09}, \cite%
{Oren10}, \cite{sodin 2007}), but also in the context of the so-called free
probability a part of the non-commutative probability theory recently
rapidly developing. One of the first papers where semicircle and related
distribution appeared in the non-commutative probability context is \cite{Bo}%
.

For $n\allowbreak <5$ distribution $f_{KMKn}$ can be identified as the
special case of the Askey--Wilson chain of distributions that make
orthogonal $5$ families of polynomials of the so-called Askey--Wilson
scheme. For the reference see e.g. \cite{Szab14}. For $n\geq 5$ the
distributions $f_{KMKn}$ were not yet described in detail.

It has to be noted that in $2009$ there appeared paper \cite{DGIX09}.
Although the aim of it was to analyze two-dimensional measures on the plane
of the form:%
\begin{equation*}
\frac{\sqrt{(1-x^{2})(1-y^{2})}}{g(x,y)},
\end{equation*}%
where $g$ was a polynomial in $x$ and $y$, it also contains some results
concerning one-dimensional case. The one-dimensional distribution considered
there, is very much alike the distribution we consider in this paper. The
authors of the paper call it Bernstein-Szeg\"{o} distribution (for
comparison see \cite{Szego39}). The method they use to analyze these
distributions allow them to consider only the case of even degree $n$ of
polynomial $g$. The results of the paper are general and hence rather
imprecise. They were obtained by quite complicated integration on the
complex plane.

Our results are precise, since we assume the exact form of the polynomial in
the denominator. Namely, we assume the knowledge of the roots of this
polynomial. Besides, as mentioned above, the order of this polynomial can be
also odd. Moreover, due to our assumptions we are able to give precise form
of the polynomials orthogonal with respect to this measure.

Our methods are simple and heavily exploit the properties of the Chebyshev
polynomials of the second kind. By using them, we are able to obtain some
interesting identities concerning Chebyshev polynomials of the second kind
as well as symmetric rational functions like for example the formulae given
in Remark \ref{gen} or Corollary \ref{identi}.

It should be mentioned that by considering polynomials of even degree,
having pairwise conjugate roots, we are not only able to cover the case of
two-dimensional measures on the plane of the form mentioned above, but also
generalize the results to $3$ or more dimensions defining new distributions
and finding they marginals. Compare Section \ref{complex} and the remarks at
the end of this section.

Hence our results give substance and generalize the results of \cite{DGIX09}.

We will present a unified approach and recall and collect information on
this family that is scattered though literature.

Let us observe first, that if $X\allowbreak \sim \allowbreak
f_{KMKn}(x|c,a_{1},\ldots ,a_{n})$ and \newline
$Y\allowbreak \sim \allowbreak f_{KMKn}(x|1,a_{1},\ldots ,a_{n}),$ then $%
X\allowbreak \sim \allowbreak cY.$ Hence we will consider further only
distributions with the density $f_{KMKn}(x|1,a_{1},\ldots ,a_{n})$ which we
will denote for brevity $f_{Kn}(x|a_{1},\ldots ,a_{n}).$

To start our analysis let us recall that for $\left\vert a_{j}\right\vert <1$
we have 
\begin{equation}
1/(1+a_{j}^{2}-2a_{j}x)\allowbreak =\allowbreak \sum_{k=0}^{\infty
}a_{j}^{k}U_{k}(x),  \label{f_gen}
\end{equation}%
where $U_{k}$ denotes $k-$th Chebyshev polynomial of the second kind and
that:\newline
\begin{equation*}
\int_{-1}^{1}\frac{2}{\pi }U_{k}(x)U_{j}(x)\sqrt{1-x^{2}}dx\allowbreak
=\allowbreak \left\{ 
\begin{array}{ccc}
0 & if & j\neq k \\ 
1 & if & j=k%
\end{array}%
\right. .
\end{equation*}%
Recall also for completeness of the paper that $U_{-1}(x)=\allowbreak 0,$ $%
U_{0}(x)\allowbreak =\allowbreak 1,$ and for $m\geq 1$ we have 
\begin{equation}
2xU_{m}(x)\allowbreak =\allowbreak U_{m+1}(x)+U_{m-1}(x).  \label{3tU}
\end{equation}
The above mentioned formula is traditionally extended to all integer $m,$
that leads to the following extended definition of the Chebyshev polynomials:%
\begin{equation}
U_{-k}(x)\allowbreak =\allowbreak U_{-k-2}(x),  \label{uu}
\end{equation}
for $k\geq 1.$

Notice that the form of the density (\ref{genK}) fits the scheme of
distributions and orthogonal polynomials that was considered in \cite{Szab14}
and hence we can use ideas and results presented there. First of all, let us
observe that the densities considered in this paper are the special cases of
the distributions known from the so-called Askey--Wilson scheme of
distributions and orthogonal polynomials obtained from the general ones by
setting $q\allowbreak =\allowbreak 0.$ $q$ is a special parameter called the
base within the theory of Askey--Wilson polynomials. More precisely, for $%
n\allowbreak =\allowbreak 0$ we deal with the so-called $q-$Hermite
polynomials and the so-called $q-$Normal distribution. Wigner distribution
and Chebyshev polynomials of the second kind are their special case for $%
q\allowbreak =\allowbreak 0$. For $n\allowbreak =\allowbreak 1$ we deal with
the so-called continuous big $q-$Hermite polynomials and the distribution
that makes them orthogonal. When we set $q\allowbreak =\allowbreak 0$ and $%
n\allowbreak =\allowbreak 1$ we deal with the distribution $f_{K1}(x|a_{1})$%
. For $n\allowbreak =\allowbreak 2$ we deal with the so-called
Al-Salam--Chihara polynomials and the distribution that make these
polynomials orthogonal. Setting now $q\allowbreak =\allowbreak 0$ leads us
to the distribution $f_{K2}(x|a_{1},a_{2})$. Further for $n\allowbreak
=\allowbreak 3$ we deal with the so-called dual Hahn polynomials and the
distribution that makes these polynomials orthogonal. Setting $q\allowbreak
=\allowbreak 0$ leads us to $f_{K3}(x|a_{1},a_{2},a_{3})$. Finally, for $%
n\allowbreak =\allowbreak 4$ we deal with the so-called Askey--Wilson
polynomials and the distribution that make them orthogonal. Setting $%
q\allowbreak =\allowbreak 0$ we get $f_{K4}(x|a_{1},a_{2},a_{3},a_{4})$.

Hence in particular, we know the families of orthogonal polynomials that our
distributions make orthogonal, for $n\allowbreak =\allowbreak 0,\ldots ,4$.
For the precise definitions and further references see e.g. \cite{KLS} or 
\cite{Szab14}.

In the sequel we will use exchangeably $\mathbf{a}_{n}$ and $(a_{1},\ldots
,a_{n})$ depending on the required brevity and clarity.

Let us denote $\prod_{j=1,j\neq i}^{n}$ as $\prod_{j\neq i}^{n}$.

We have the following general observation.

\begin{theorem}
If $a_{i}\neq a_{j},$ $i\neq j,$ $i,j\allowbreak =\allowbreak 1,\ldots ,n,$
then \newline
i) the constant $A_{n}$ in (\ref{genK}) is given by: 
\begin{equation}
A_{n}\allowbreak =A_{n}(\mathbf{a}_{n})=\allowbreak 1/\sum_{i=1}^{n}\frac{%
a_{i}^{n-1}}{\prod_{j\neq i}^{n}(a_{i}-a_{j})(1-a_{i}a_{j})}.  \label{stala}
\end{equation}%
ii) Let us define constants $B_{n,k}$ for $n\,<0$ as $B_{n,k}\allowbreak
=\allowbreak 0$ and for $n,k\geq 0$ by:%
\begin{equation*}
B_{n,k}(\mathbf{a}_{n})=B_{n,k}=\int_{-1}^{1}U_{k}(x)f_{Kn}(x|\mathbf{a}%
_{n})dx,
\end{equation*}%
and for $n\geq 0$ and $k\leq -1$ by: 
\begin{equation}
B_{n,k}=\left\{ 
\begin{array}{ccc}
0 & if & k=-1 \\ 
-B_{n,-k-2} & if & k\leq -2%
\end{array}%
\right. .  \label{-b}
\end{equation}%
Then we have: 
\begin{equation}
B_{n,k}=A_{n}\sum_{i=1}^{n}\frac{a_{i}^{n+k-1}}{\prod_{j\neq
i}^{n}(a_{i}-a_{j})(1-a_{i}a_{j})},  \label{roz}
\end{equation}%
for $k\allowbreak =\allowbreak 0,...$.
\end{theorem}

\begin{proof}
Firstly, notice that the definition of $B_{n,k}$ for the negative integers $k
$ follows the fact that we have the identity (\ref{uu}), above. We will use
the fact that 
\begin{equation*}
1/\prod_{i=1}^{n}(x-b_{i})\allowbreak =\allowbreak \sum_{i=1}^{n}\frac{1}{%
\prod_{j=1,j\neq i}^{n}(b_{i}-b_{j})}\frac{1}{(x-b_{i})}
\end{equation*}%
and the fact that $((1+a^{2})/a-(1+b^{2})/b)\allowbreak =\allowbreak
(b-a)(1-ab)/(ab).$ Hence, we have%
\begin{gather*}
f_{Kn}(x|a_{1},\ldots ,a_{n})\allowbreak =\allowbreak \frac{2A_{n}(-1)^{n}}{%
2^{n}\pi \prod_{i=1}^{n}a_{i}}\sqrt{1-x^{2}}%
/\prod_{i=1}^{n}(x-(1+a_{i}^{2})/(2a_{i})) \\
=\frac{2\times 2^{n-1}A_{n}(-1)^{n}}{2^{n}\pi \prod_{i=1}^{n}a_{i}}\sqrt{%
1-x^{2}}\sum_{i=1}^{n}\frac{a_{i}^{n-1}\prod_{j\neq i}^{n}a_{j}}{%
\prod_{j\neq i}^{n}(a_{j}-a_{i})(1-a_{i}a_{j})}\frac{1}{%
(x-(1+a_{i}^{2})/(2a_{i}))} \\
=\frac{2A_{n}}{\pi }\sqrt{1-x^{2}}\sum_{i=1}^{n}\frac{a_{i}^{n-1}}{%
\prod_{j\neq i}^{n}(a_{i}-a_{j})(1-a_{i}a_{j})}\frac{1}{%
((1+a_{i}^{2})-2a_{i}x)}.
\end{gather*}%
Now following (\ref{f_gen}) and orthogonality of polynomials $U^{\prime }s$
with respect to Wigner measure we have 
\begin{equation*}
\int_{-1}^{1}\frac{2\sqrt{1-x^{2}}}{\pi ((1+a_{i}^{2})-2a_{i}x)}%
U_{k}(x)dx\allowbreak =\allowbreak a_{i}^{k}.
\end{equation*}
\end{proof}

\begin{remark}
Since coefficients $B_{n,k}$ are the coefficients in the orthogonal
expansion in $L_{2}([-1,1],w),$ in the basis $\left\{ U_{k}\right\} _{k\geq
0},$ where $w$ denotes measure with the density $\frac{2}{\pi }\sqrt{1-x^{2}}
$, we get the following expansion for free:%
\begin{equation}
f_{Kn}(x|\mathbf{a}_{n})\allowbreak =\allowbreak \frac{2}{\pi }\sqrt{1-x^{2}}%
\sum_{k=0}^{\infty }B_{n,k}U_{k}(x).  \label{expp}
\end{equation}%
For more examples of such expansions see \cite{Szablowski2010(1)}.
\end{remark}

In the sequel we will need the following quantities:%
\begin{eqnarray}
S_{k}^{(n)}(\mathbf{a}_{n}) &=&\sum_{1\leq j_{1}<j_{2}<\ldots ,<j_{k}\leq
n}\prod_{m=1}^{k}a_{j_{m}},  \label{sym} \\
\Delta _{m}^{(n)}(\mathbf{a}_{n}) &=&\sum_{\substack{ 0\leq j_{1},\ldots
j_{n-1}\leq m  \\ j_{1}+\ldots +j_{n-1}\leq m}}a_{1}^{j_{1}}a_{2}^{j_{2}}%
\ldots a_{n-1}^{j_{n-1}}a_{n}^{m-(j_{1}+\ldots +j_{n-1})},
\end{eqnarray}%
for $1\leq k\leq n\leq m$. Whenever it will be obvious we will drop other
than $k$ arguments of the functions $S_{k}$ and $\Delta _{k}.$ Notice that $%
S_{k}^{(n)}$ is the $k-$th elementary symmetric function of the variables $%
a_{1},\ldots ,a_{n}.$ We set $S_{0}^{(n)}(\mathbf{a}_{n})\allowbreak
=\allowbreak 1$ and $S_{k}^{(n)}(\mathbf{a}_{n})\allowbreak =\allowbreak 0$
when $k>n>0$.

\begin{remark}
\label{part}We have 
\begin{gather*}
A_{1}\allowbreak =\allowbreak 1,A_{2}\allowbreak =\allowbreak 1-a_{1}a_{2},~
\\
A_{3}\allowbreak =\allowbreak
\prod_{j=1}^{3}%
\prod_{k=j+1}^{3}(1-a_{j}a_{k})=(1-a_{1}a_{2})(1-a_{1}a_{3})(1-a_{2}a_{3}),
\\
A_{4}\allowbreak =\allowbreak
\prod_{j=1}^{4}\prod_{k=j+1}^{4}(1-a_{j}a_{k})/(1-S_{4})= \\
\frac{%
(1-a_{1}a_{2})(1-a_{1}a_{3})(1-a_{2}a_{3})(1-a_{1}a_{4})(1-a_{2}a_{3})(1-a_{2}a_{4})(1-a_{3}a_{4})%
}{1-a_{1}a_{2}a_{3}a_{4}}, \\
A_{5}\allowbreak =\allowbreak
\prod_{j=1}^{5}\prod_{k=j+1}^{5}(1-a_{j}a_{k})/(1\allowbreak -\allowbreak
S_{4}\allowbreak +\allowbreak S_{1}S_{5}\allowbreak -\allowbreak S_{5}^{2}),
\\
A_{6}\allowbreak =\allowbreak \\
\frac{\prod_{j=1}^{6}\prod_{k=j+1}^{6}(1-a_{j}a_{k})}{(1-S_{4}\allowbreak
+\allowbreak S_{1}S_{5}\allowbreak -\allowbreak S_{5}^{2}\allowbreak
-S_{6}\allowbreak -\allowbreak S_{1}^{2}S_{6}\allowbreak +\allowbreak
S_{2}S_{6}\allowbreak \allowbreak +\allowbreak S_{4}S_{6}\allowbreak
+\allowbreak S_{1}S_{5}S_{6}\allowbreak -\allowbreak S_{6}^{2}\allowbreak
-\allowbreak S_{2}S_{6}^{2}\allowbreak +\allowbreak S_{6}^{3})}.
\end{gather*}%
\newline
Notice that for $n\allowbreak =\allowbreak 1,\ldots ,4$ the constant $A_{n}$
agrees with respective constants given in \cite{Szab14} in formulae (2.4),
(2.6) and unnamed formulae on top of 9-th and 10-th pages for $q=0,$ when
presenting respectively densities that make big $q$-Hermite,
Al-Salam--Chihara, dual Hahn and Askey-Wilson polynomials orthogonal.

Using this denotations we have $B_{2,k}\allowbreak =\allowbreak \Delta
_{k}^{(2)},$ $B_{3,k}\allowbreak =\allowbreak \Delta _{k}^{(3)}-S_{3}\Delta
_{k-1}^{(3)}.$ Compare here with formulae (2.7) and (2.10) in \cite{Szab14}
for $q\allowbreak =\allowbreak 0.$
\end{remark}

\begin{remark}
\label{gen}Let us notice also that following (\ref{f_gen}) we have: 
\begin{eqnarray*}
1/A_{n}(\mathbf{a}_{n})\allowbreak  &=&\allowbreak \int_{-1}^{1}\frac{2}{\pi 
}\sqrt{1-x^{2}}\prod_{i=1}^{n}\sum_{k_{i}=1}^{\infty
}a_{i}^{k_{i}}U_{k_{i}}(x)\allowbreak dx\allowbreak = \\
&&\sum_{k_{1}\geq 0,\ldots ,k_{n}\geq
0}(\prod_{j=1}^{n}a_{i}^{k_{i}})V_{k_{1},\ldots ,k_{n}},
\end{eqnarray*}%
$\allowbreak $ where $V_{k_{1},\ldots ,k_{n}}\allowbreak =\allowbreak
\int_{-1}^{1}\frac{2}{\pi }\sqrt{1-x^{2}}\prod_{i=1}^{n}U_{k_{i}}(x)dx$.
That is, we get the generating function of numbers $V_{k_{1},\ldots ,k_{n}}$
for free. This observation gives also an interesting interpretation of the
constants $A_{n}.$
\end{remark}

\begin{corollary}
i) 
\begin{equation*}
\int_{-1}^{1}x^{k}f_{Kn}(x|\mathbf{a}_{n})dx\allowbreak =\allowbreak \frac{1%
}{(k+1)2^{k}}\sum_{j=0}^{\left\lfloor k/2\right\rfloor }(k-2j+1)\binom{k+1}{j%
}B_{n,k-2j}(\mathbf{a}_{n}),
\end{equation*}

ii) 
\begin{equation*}
\int_{-1}^{1}U_{k}(x)U_{m}(x)f_{Kn}(x|\mathbf{a}_{n})dx\allowbreak
=\allowbreak \sum_{j=0}^{\min (m,k)}B_{n,\left\vert m-k\right\vert +2j}(%
\mathbf{a}_{n}).
\end{equation*}
\end{corollary}

\begin{proof}
i) We use well known identity: 
\begin{equation*}
2^{k}x^{k}\allowbreak =\allowbreak \sum_{j=0}^{\left\lfloor k/2\right\rfloor
}(\binom{k}{j}-\binom{k}{j-1})U_{k-2j}(x)
\end{equation*}%
(see e.g. \cite{Szab15} Proposition 1 with $q=0$ and the fact that $%
h_{n}(x|0)\allowbreak =\allowbreak U_{n}(x)$) and the fact that 
\begin{equation*}
\binom{k}{j}-\binom{k}{j-1}\allowbreak =\allowbreak (k-2j+1)\binom{k+1}{j}%
/(k+1).
\end{equation*}%
ii) We use identity 
\begin{equation*}
U_{k}(x)U_{m}(x)\allowbreak =\allowbreak \sum_{j=0}^{\min
(k,m)}U_{\left\vert k-m\right\vert +2j}(x),
\end{equation*}%
that can be easily derived from 
\begin{equation*}
2T_{n}(x)T_{m}(x)=T_{n+m}(x)+T_{\left\vert n-m\right\vert }(x),
\end{equation*}%
where $T_{n}$ denotes Chebyshev polynomial of the first kind and the
formulae relating Chebyshev polynomials of the first and second kind. See
also formula (2.13) of \cite{Szab-rev} with $q=0.$
\end{proof}

\begin{proposition}
For any function $g:\mathbb{R\longrightarrow \mathbb{R}}$ let us denote $%
\mathbf{g(a}_{n})\allowbreak =$\newline
$\allowbreak (g(a_{1}),\ldots ,g(a_{n}))$ and also $\mathbf{b}%
_{n}^{(i)}\allowbreak =\allowbreak (b_{1},\ldots ,b_{i-1},b_{i+1},\ldots
,b_{n})$. Let us set $g(x)\allowbreak =(1+x^{2})/(2x)$. Then, if $a_{i}\neq
a_{j},$ $i\neq j,$ $i,j\allowbreak =\allowbreak 1,\ldots ,n,$ we have $:$ 
\begin{equation}
\sum_{i=1}^{n}\frac{a_{i}^{n-2}S_{k}(\mathbf{g(a}_{n}^{(i)}))}{\prod_{j\neq
i}^{n}(a_{j}-a_{i})(1-a_{i}a_{j})}\allowbreak =\left\{ 
\begin{array}{ccc}
0 & \text{for} & k\allowbreak =0,\ldots ,n-2, \\ 
1 & \text{for} & k=n-1,%
\end{array}%
\right. \allowbreak 0.  \label{id}
\end{equation}
\end{proposition}

\begin{proof}
We start from the fact that $1/\prod_{i=1}^{n}(x-b_{i})\allowbreak
=\allowbreak \sum_{i=1}^{n}D_{n,i}\frac{1}{(x-b_{i})}$, where $%
D_{n,i}\allowbreak =\allowbreak 1/\prod_{j=1,j\neq i}^{n}(b_{i}-b_{j}).$
Thus, from the properties of simple fraction decompositions we have the
following identity: 
\begin{equation*}
1/\prod_{i=1}^{n}(x-b_{i})\allowbreak =\allowbreak
(\sum_{i=1}^{n}D_{n,i}\prod_{j\neq i}(x-b_{j}))/\prod_{i=1}^{n}(x-b_{i}).
\end{equation*}%
Hence all coefficients in $\sum_{i=1}^{n}D_{n,i}\prod_{j\neq i}(x-b_{j}))$
by nonzero powers of $x$ must be zero. In particular for $\forall
k\allowbreak =\allowbreak 0\ldots ,n-1$ we must have $%
\sum_{i=1}^{n}D_{n,i}S_{k}(\mathbf{b}_{n}^{(i)})\allowbreak =\allowbreak 0$.
Now it remains to substitute $b_{i}\allowbreak =\allowbreak g(a_{i})$ and
use the fact that $(b_{i}-b_{j})\allowbreak =\allowbreak
(a_{j}-a_{i})(1-a_{i}a_{j})/(2a_{i}a_{j})$ so that 
\begin{equation*}
D_{n,i}\allowbreak =\allowbreak
2^{n-1}a_{i}^{n-2}(\prod_{j=1}^{n}a_{j})/\prod_{j\neq
i}^{n}((1-a_{i}a_{j})(a_{j}-a_{i})),
\end{equation*}%
since we have $a_{i}^{n-1}\prod_{j\neq i}^{n}a_{j}\allowbreak =\allowbreak
a_{i}^{n-2}\prod_{j=1}^{n}a_{j}$.
\end{proof}

As a corollary we get in particular the following identities.

\begin{corollary}
\label{identi}$\forall n\geq 2:$ i)%
\begin{equation}
\sum_{i=1}^{n}\frac{a_{i}^{n-2}}{\prod_{j\neq
i}^{n}(a_{i}-a_{j})(1-a_{i}a_{j})}\allowbreak =\allowbreak 0,  \label{an-2}
\end{equation}%
from which it follows that (\ref{roz}) valid also for $k\allowbreak
=\allowbreak -1,$

ii) 
\begin{equation*}
\sum_{i=1}^{n}\frac{a_{i}^{n-2}(1+a_{i}^{2})\sum_{j\neq i}^{n}\prod_{k\neq
i}^{n}a_{k}}{\prod_{j\neq i}^{n}(a_{i}-a_{j})(1-a_{i}a_{j})}\allowbreak
=\allowbreak 0,
\end{equation*}

iii) $\forall m\geq n-1,n\geq 1:$%
\begin{equation}
\sum_{j=0}^{n}(-1)^{j}S_{j}^{(n)}(\mathbf{a}_{n})B_{n,m-j}(\mathbf{a}%
_{n})\allowbreak =\allowbreak 0.  \label{id2}
\end{equation}
\end{corollary}

\begin{proof}
To prove i) we take $k\allowbreak =\allowbreak 0$ in (\ref{id}). To prove
ii) we take $k\allowbreak =\allowbreak 1$ in (\ref{id}) and notice that 
\begin{eqnarray*}
a_{i}^{n-1}(\prod_{j\neq i}^{n}a_{j})S_{1}(\mathbf{g(a}_{n}^{(i)})%
\allowbreak &=&\allowbreak a_{i}^{n-1}\sum_{j\neq
i}^{n}(1+a_{j}^{2})\prod_{k\neq i,j}^{n}a_{k}\allowbreak = \\
\allowbreak a_{i}^{n-2}\sum_{j\neq i}^{n}\prod_{k\neq j}a_{k}\allowbreak
+\allowbreak a_{i}^{n}\sum_{j\neq i}^{n}\prod_{k\neq j}a_{k}\allowbreak
&=&\allowbreak a_{i}^{n-2}(1+a_{i}^{2})\sum_{j\neq i}^{n}\prod_{k\neq
j}^{n}a_{k}.
\end{eqnarray*}

To prove iii) we observe first that for $n\geq 1,$ $k\geq 0,\forall
i\allowbreak =\allowbreak 1,\ldots n$ we have%
\begin{equation}
\sum_{j=0}^{n}(-1)^{j}a_{i}^{n+k-j}S_{j}^{(n)}(\mathbf{a}_{n})\allowbreak
=\allowbreak 0,  \label{el_id}
\end{equation}%
which is elementary to notice. Secondly, using (\ref{roz}) for $m\geq n-1,$
we get 
\begin{gather*}
\sum_{k=0}^{n}(-1)^{k}S_{k}^{(n)}B_{n,m-k}\allowbreak =\allowbreak
A_{n}\sum_{k=0}^{n}(-1)^{k}S_{k}^{(n)}\sum_{i=1}^{n}\frac{a_{i}^{n+m-k-1}}{%
\prod_{j\neq i}^{n}(a_{i}-a_{j})(1-a_{i}a_{j})} \\
=A_{n}\sum_{i=1}^{n}\frac{1}{\prod_{j\neq i}^{n}(a_{i}-a_{j})(1-a_{i}a_{j})}%
\sum_{k=0}^{n}(-1)^{k}S_{k}^{(n)}a_{i}^{n+m-k-1}=0,
\end{gather*}%
by (\ref{el_id}).
\end{proof}

\begin{theorem}
\label{p_ort} For every $m\geq 2n-2\geq 0$ the family of polynomials
orthogonal with respect to $f_{Kn}$ is of the form:%
\begin{equation}
P_{m}^{(n)}(x|\mathbf{a}_{n})=\sum_{j=0}^{n-1}(-1)^{j}U_{m-j}(x)S_{j}^{(n)}(%
\mathbf{a}_{n}),  \label{p_o}
\end{equation}%
where $S_{j}^{(n)}\allowbreak $ is given by (\ref{sym}) . $\,$
\end{theorem}

\begin{proof}
The fact that the polynomials $P_{m}^{(n)}$ can be expressed as linear
combination of the last at most $n+1$ polynomials $U_{m}$ follows directly
from \cite{Szablowski2010(1)}(Proposition 1 iii)) or from \cite{DGIX09}
(Lemma 3.1). Similar fact was noticed for $n\allowbreak =\allowbreak 1,2$
earlier by Maroni in his papers published in the 90's. Next notice that $%
P_{m}^{(n)}$ must be of the form $U_{m}+\sum_{j=1}^{n}b_{j}^{(m)}U_{m-j}(x).$
To determine parameters $\left\{ b_{j}^{(m)}\right\} _{j=1}^{n-1}$ we need $%
n $ equation of the form $\int_{-1}^{1}x^{k}P_{m}^{(n)}(x|\mathbf{a}%
_{n})f_{Kn}(x|\mathbf{a}_{n})dx\allowbreak =\allowbreak 0,$ $k\allowbreak
=\allowbreak 0,\ldots ,n-1.$

Now notice that for $m\geq n-1$ we have 
\begin{gather}
\int_{-1}^{1}P_{m}^{(n)}(x|\mathbf{a}_{n})f_{Kn}(x|\mathbf{a}_{n})dx=0,
\label{inte} \\
2xP_{m}^{(n)}(x|\mathbf{a}_{n})\allowbreak =\allowbreak P_{m+1}^{(n)}(x|%
\mathbf{a}_{n})+P_{m-1}^{(n)}(x|\mathbf{a}_{n}).  \label{3r}
\end{gather}%
which in the case of (\ref{inte}) follows (\ref{roz}), and (\ref{id2}) and
in case of (\ref{3r}) follows directly three term recurrence for the
Chebyshev polynomials. More over iterating (\ref{3r}) we can express $%
x^{k}P_{m}^{(n)}(x|\mathbf{a}_{n})$ as linear combination of $P_{m+l}^{(n)},$
for $l\allowbreak =\allowbreak -k,\ldots ,k.$ Since we have to have $m-k\geq
n-1,$ for $k\allowbreak =\allowbreak 0,\ldots ,n-1$ we see that polynomials $%
P_{m}^{(n)}$ orthogonal for $m\geq 2n-2.$
\end{proof}

\begin{remark}
Notice that polynomials $P_{m}^{(n)}(x|\mathbf{a}_{n})/2^{m}$ are monic,
since $U_{n}/2^{n}$ are monic for $n\geq 0$.
\end{remark}

\begin{remark}
Recall that the first Askey-Wilson polynomials $aw_{k}(x|\mathbf{a}_{4})$
with $q\allowbreak =\allowbreak 0$ are equal to: 
\begin{eqnarray*}
aw_{1}(x|\mathbf{a}_{4}) &=&U_{1}(x)-\frac{S_{1}-S_{3}}{1-S_{4}}, \\
aw_{2}(x|\mathbf{a}_{4}) &=&U_{2}(x)-S_{1}U_{1}(x)+S_{2}-S_{4}, \\
aw_{3}(x|\mathbf{a}_{4}) &=&\sum_{j=0}^{3}(-1)^{j}U_{k-j}(x)S_{j}, \\
aw_{k}(x|\mathbf{a}_{4}) &=&\sum_{j=0}^{4}(-1)^{j}U_{k-j}(x)S_{j},
\end{eqnarray*}%
$k\allowbreak \geq 4,$ where, as agreed above, $S_{j}$ means in fact $%
S_{j}^{(4)}(\mathbf{a}_{4}).$ From this presentation it follows that may be
the formula (\ref{p_o}) is valid for $m\geq n-1.$ In fact, there is a strong
argument to support this supposition. Namely numerical simulations suggest
that (\ref{id2}) might be true for $m\geq 1.$ Notice that it is impossible
to further extend this formula, i.e. to fit it for the cases $m<n-1.$ This
is so since immediately we see that $P_{1}^{(n)}(x|\mathbf{a}%
_{n})\allowbreak =\allowbreak U_{1}(x)-B_{n,1}(\mathbf{a}_{n})$ and $B_{n,1}(%
\mathbf{a}_{n})$ is different from $S_{1}$ for $n\geq 3$ since we have:%
\begin{eqnarray*}
B_{3,1} &=&S_{1}-S_{3},~B_{4,1}\allowbreak =\allowbreak \frac{S_{1}-S_{3}}{%
1-S_{4}}, \\
B_{5,1}\allowbreak &=&\allowbreak \frac{S_{1}-S_{3}+S_{1}S_{5}-S_{4}S_{5}}{%
1-S_{4}+S_{1}S_{5}-S_{5}^{2}}.
\end{eqnarray*}
\end{remark}

\begin{proposition}
\label{cauchy}Let $X\allowbreak \sim \allowbreak f_{Kn}(x|\mathbf{a}_{n})$
then for $\forall \left\vert t\right\vert <1$ we have%
\begin{equation*}
E\frac{1}{1+t^{2}-2tX}\allowbreak =\allowbreak \frac{Q_{n}(t|\mathbf{a}_{n})%
}{\prod_{i=1}^{n}(1-ta_{i})},
\end{equation*}%
where 
\begin{equation*}
Q_{n}(t|\mathbf{a}_{n})\allowbreak =\allowbreak
A_{n}\sum_{i=1}^{n}a_{i}^{n-1}\prod_{j\neq i}^{n}\frac{(1-a_{j}t)}{%
(a_{i}-a_{j})(1-a_{i}a_{j})}
\end{equation*}%
is a polynomial of degree $\max (n-2,0)$ in $t.$
\end{proposition}

\begin{proof}
Recall that $\sum_{j=0}^{\infty }t^{j}U_{j}(x)\allowbreak =\allowbreak
1/(1+t^{2}-2tx).$ Thus 
\begin{gather*}
\int_{-1}^{1}(\sum_{j=0}^{\infty }t^{j}U_{j}(x))f_{Kn}(x|\mathbf{a}%
_{n})dx\allowbreak =\allowbreak A_{n}\sum_{j=0}^{\infty }t^{j}\sum_{i=1}^{n}%
\frac{a_{i}^{n+j-1}}{\prod_{j\neq i}^{n}(a_{i}-a_{j})(1-a_{i}a_{j})}%
\allowbreak \\
=\allowbreak A_{n}\sum_{i=1}^{n}\frac{a_{i}^{n-1}}{\prod_{j\neq
i}^{n}(a_{i}-a_{j})(1-a_{i}a_{j})}\sum_{j\geq 0}t^{j}a_{i}^{j}\allowbreak
=\allowbreak A_{n}\sum_{i=1}^{n}\frac{1}{1-a_{i}t}\frac{a_{i}^{n-1}}{%
\prod_{j\neq i}^{n}(a_{i}-a_{j})(1-a_{i}a_{j})}\allowbreak \\
=\allowbreak \frac{1}{\prod_{i=1}^{n}(1-ta_{i})}\allowbreak A_{n}\allowbreak
\sum_{i=1}^{n}a_{i}^{n-1}\prod_{j\neq i}^{n}\frac{(1-a_{j}t)}{%
(a_{i}-a_{j})(1-a_{i}a_{j})}.
\end{gather*}%
$\allowbreak $The fact that $Q_{n}$ is a polynomial of degree $n-2$ follows
the fact that $a_{i}^{n-1}\prod_{j\neq i}^{n}(1-a_{j}t)$ is a polynomial of
degree $n-1,$ but the coefficient by $t^{n-1}$ is equal to $S_{n}(\mathbf{a}%
_{n})a_{i}^{n-2.}$. The assertion follows from (\ref{an-2}).
\end{proof}

\begin{remark}
Recall that the Cauchy transform of a measure $\mu $ is defined by: 
\begin{equation*}
\mathcal{C}_{\mu }(y)\allowbreak =\allowbreak \int \frac{d\mu (x)}{y-x},
\end{equation*}%
where the integral is understood as a principal value. Note that the names
Hilbert or Stjeltjes transform are also used. I has been intensively studied
recently in connection with free probability or signal processing. For the
reference see \cite{Cima06} and also papers of D. Voiculescu and his fellow
researchers as well as of B. Shapiro and his fellow researchers. Notice also
that Proposition \ref{cauchy} helps to get values of the Cauchy transform of 
$f_{Kn}(x|\mathbf{a}_{n})$ for real $\left\vert y\right\vert >1$ by taking $%
y\allowbreak =\allowbreak (1+t^{2})/2t,$ $\left\vert t\right\vert <1$. More
precisely, we have%
\begin{equation*}
\mathcal{C}_{Kn}((1+t^{2})/2t)=2t\frac{Q_{n}(t|\mathbf{a}_{n})}{%
\prod_{i=1}^{n}(1-ta_{i})}.
\end{equation*}
\end{remark}

\begin{remark}
By direct calculations we have 
\begin{eqnarray*}
Q_{1}(t|a)\allowbreak &=&\allowbreak 1,\text{ }Q_{2}(t|\mathbf{a}%
_{2})\allowbreak =\allowbreak 1,\text{ }Q_{3}(t|\mathbf{a}_{3})\allowbreak
=\allowbreak 1-tS_{3}(\mathbf{a}_{3}), \\
Q_{4}(t|\mathbf{a}_{4})\allowbreak &=&\allowbreak ((1-S_{4})\allowbreak
-\allowbreak t(S_{3}-S_{1}S_{4})\allowbreak +\allowbreak
t^{2}S_{4}(1-S_{4}))/(1-S_{4}).
\end{eqnarray*}
\end{remark}

As an immediate consequence of this formula (\ref{f_gen}) and the definition
of $B_{n,k}$ we get characteristic functions of numbers $B_{n,k}(\mathbf{a}%
_{n})$.

\begin{corollary}
For $\forall n\geq 0$ we have%
\begin{equation*}
\sum_{k\geq 0}t^{k}B_{n,k}(\mathbf{a}_{n})\allowbreak =\allowbreak Q_{n}(t|%
\mathbf{a}_{n})/\prod_{i=1}^{n}(1-ta_{i}).
\end{equation*}
\end{corollary}

\section{Complex parameters\label{complex}}

In this section we will study properties of the generalized Kesten-MacKay
distributions for even $n=2k$ and parameters $a_{i}$ , $i\allowbreak
=\allowbreak 1,\ldots ,2k$ being complex and forming conjugate pairs. The
new parameters will have new names, namely for the conjugate the pair for
example $a_{i}\allowbreak =\allowbreak \rho _{i}\exp (i\theta _{i})$ and $%
a_{k+i}\allowbreak =\allowbreak \rho _{i}\exp (-i\theta _{i})$ we will
denote $y_{i}\allowbreak =\allowbreak \cos \theta _{i}$ for $i\allowbreak
=\allowbreak 1,\ldots ,k$. Besides notice that we have%
\begin{eqnarray*}
&&(1+\rho _{i}^{2}\exp (2i\theta _{i})-2x\rho _{i}\exp (i\theta
_{i}))(1+\rho _{i}^{2}\exp (-2i\theta _{i})-2x\rho _{i}\exp (-i\theta
_{i}))\allowbreak \\
&=&1+\rho _{i}^{4}+4x^{2}\rho _{i}^{2}-4x\rho _{i}(1+\rho _{i}^{2})\cos
\theta _{i}+2\rho _{i}^{2}\cos 2\theta _{i}\allowbreak =w(x,y_{i}|\rho _{i}),
\end{eqnarray*}%
where we denoted for simplicity:%
\begin{equation}
w(x,y|\rho )\allowbreak =\allowbreak (1-\rho ^{2})^{2}-4xy\rho (1+\rho
^{2})+4\rho ^{2}(x^{2}+y^{2}).  \label{ww}
\end{equation}

Hence now the density $f_{K2k}(x|a_{1},\ldots ,a_{2k})$ will have the
following form that we will denote by $f_{Mk}(x|\mathbf{y}_{k},\mathbf{\rho }%
_{k}):$%
\begin{equation}
f_{Mk}(x|\mathbf{y}_{k},\mathbf{\rho }_{k})\allowbreak =\allowbreak A_{2k}%
\frac{2\sqrt{1-x^{2}}}{\pi \prod_{j=1}^{k}w(x,y_{i}|\rho _{i})},
\label{dens}
\end{equation}%
with $\left\vert \rho _{i}\right\vert <1,$ $\left\vert y_{i}\right\vert \leq
1.$

Following Remark \ref{part} we have

\begin{lemma}
\label{szcz}i) $S_{1}^{(2)}\allowbreak =\allowbreak 2\rho y$, $%
S_{2}^{(2)}=\allowbreak \rho ^{2}$, $A_{2}\allowbreak =\allowbreak 1-\rho
^{2},$

ii) 
\begin{gather*}
S_{1}^{(4)}(\mathbf{a}_{4})\allowbreak =\allowbreak y_{1}\rho
_{1}\allowbreak +\allowbreak y_{2}\rho _{2}, \\
S_{2}^{(4)}(\mathbf{a}_{4})\allowbreak =\allowbreak \rho _{1}^{2}\allowbreak
+\allowbreak \rho _{2}^{2}\allowbreak +\allowbreak 4y_{1}y_{2}\rho _{1}\rho
_{2}, \\
S_{3}^{(4)}(\mathbf{a}_{4})\allowbreak =\allowbreak 2\rho _{1}\rho _{2}(\rho
_{1}y_{1}\allowbreak +\allowbreak \rho _{2}y_{2}), \\
A_{4}\allowbreak =\allowbreak (1-\rho _{1}^{2})(1-\rho
_{2}^{2})w(y_{1},y_{2}|\rho _{1}\rho _{2})/(1-\rho _{1}^{2}\rho _{2}^{2}).
\end{gather*}

iii) 
\begin{gather*}
S_{1}^{(6)}(\mathbf{a}_{6})\allowbreak =\allowbreak \rho
_{1}y_{1}\allowbreak +\allowbreak \rho _{2}y_{2}\allowbreak +\allowbreak
\rho _{3}y_{3}, \\
S_{2}^{(6)}(\mathbf{a}_{6})\allowbreak =\allowbreak \rho _{1}^{2}\allowbreak
+\allowbreak \rho _{2}^{2}\allowbreak +\allowbreak \rho _{3}^{2}\allowbreak
+\allowbreak 4(\rho _{1}\rho _{2}y_{1}y_{2}\allowbreak +\allowbreak \rho
_{1}\rho _{3}y_{1}y_{3}\allowbreak +\allowbreak \rho _{2}\rho
_{3}y_{2}y_{3}),
\end{gather*}%
\newline
\begin{eqnarray*}
S_{3}^{(6)}(\mathbf{a}_{6})\allowbreak &=&\allowbreak 2(\rho _{1}^{2}+\rho
_{3}^{2})\rho _{2}y_{2}\allowbreak +\allowbreak 2(\rho _{2}^{2}+\rho
_{3}^{2})\rho _{1}y_{1}\allowbreak \\
&&+\allowbreak 2(\rho _{1}^{2}+\rho _{2}^{2})\rho _{3}y_{3}\allowbreak
+\allowbreak 8\rho _{1}\rho _{2}\rho _{3}y_{1}y_{2}y_{3},
\end{eqnarray*}%
$\allowbreak $ \newline
\begin{equation*}
S_{4}^{(6)}(\mathbf{a}_{6})\allowbreak =\allowbreak \rho _{1}^{2}\rho
_{2}^{2}\allowbreak +\allowbreak \rho _{2}^{2}\rho _{3}^{2}\allowbreak
+\allowbreak \rho _{1}^{2}\rho _{3}^{2}\allowbreak +\allowbreak 4\rho
_{1}\rho _{2}\rho _{3}(\rho _{3}y_{1}y_{2}\allowbreak +\rho
_{2}y_{1}y_{3}\allowbreak +\allowbreak \rho _{1}y_{2}y_{3}),
\end{equation*}%
\newline
\begin{eqnarray*}
S_{5}^{(6)}(\mathbf{a}_{6})\allowbreak &=&\allowbreak 2\rho _{1}\rho
_{2}\rho _{3}(\rho _{1}\rho _{2}y_{3}\allowbreak +\allowbreak \rho _{1}\rho
_{3}y_{2}\allowbreak +\allowbreak \rho _{2}\rho _{3}y_{1}),~S_{6}^{(6)}(%
\mathbf{a}_{6})\allowbreak =\allowbreak \rho _{1}^{2}\rho _{2}^{2}\rho
_{3}^{2}, \\
A_{6}\allowbreak &=&\allowbreak (1-\rho _{1}^{2})(1-\rho _{2}^{2})(1-\rho
_{3}^{2})\frac{w(y_{1},y_{2}|\rho _{1}\rho _{2})w(y_{2},y_{3}|\rho _{2}\rho
_{3})w(y_{1},y_{3}|\rho _{1}\rho _{3})}{w3(y_{1},y_{2},y_{3}|\rho _{1},\rho
_{2},\rho _{3})}\allowbreak ,
\end{eqnarray*}%
where we denoted%
\begin{gather}
w3(y_{1},y_{2},y_{3}|\rho _{1},\rho _{2},\rho _{3})=(1-\rho _{1}^{2}\rho
_{2}^{2})(1-\rho _{2}^{2}\rho _{3}^{3})(1-\rho _{1}^{2}\rho _{3}^{2})(1-\rho
_{1}^{2}\rho _{2}^{2}\rho _{3}^{2})  \label{w3} \\
-4\rho _{1}\rho _{2}\rho _{3}(1+\rho _{1}^{2}\rho _{2}^{2}\rho
_{3}^{2})(\rho _{1}(1-\rho _{2}^{2})(1-\rho _{3}^{2})y_{2}y_{3}\allowbreak 
\notag \\
+\allowbreak \rho _{2}(1-\rho _{1}^{2})(1-\rho
_{3}^{2})y_{1}y_{3}\allowbreak +\allowbreak \rho _{3}(1-\rho
_{1}^{2})(1-\rho _{2}^{2})y_{1}y_{2})  \notag \\
+4\rho _{1}^{2}\rho _{2}^{2}\rho _{3}^{2}((1-\rho _{1}^{2})(1-\rho
_{2}^{2}\rho _{3}^{2})y_{1}^{2}\allowbreak +\allowbreak (1-\rho
_{2}^{2})(1-\rho _{1}^{2}\rho _{3}^{2})y_{2}^{2}\allowbreak +\allowbreak
(1-\rho _{3}^{2})(1-\rho _{1}^{2}\rho _{2}^{2})y_{3}^{2}).  \notag
\end{gather}
\end{lemma}

\begin{proof}
All calculations were done using Mathematica 10.
\end{proof}

\begin{remark}
\label{bivariate}Recall that for all $\left\vert x\right\vert ,\left\vert
y\right\vert \leq 1$ and $\left\vert \rho \right\vert <1$ we have the
following useful expansion: 
\begin{equation}
\frac{1-\rho ^{2}}{w(x,y|\rho )}=\sum_{j=0}^{\infty }\rho
^{j}U_{j}(x)U_{j}(y),  \label{P-M}
\end{equation}%
which is nothing else but the famous Poisson--Mehler formula for $%
q\allowbreak =\allowbreak 0$ (for the reference, see e.g. \cite{IA}
(13.1.24) or for alternative proof \cite{Szablowski2010(1)}).
\end{remark}

Following the above mentioned remark we have%
\begin{gather}
f_{Mk}(x|\mathbf{y}_{k},\mathbf{\rho }_{k})=\frac{2A_{2k}}{\pi
\prod_{j=1}^{k}(1-\rho _{j}^{2})}\sqrt{1-x^{2}}\sum_{m_{1},m_{2},\ldots
m_{k}=0}^{\infty }\prod_{j=1}^{k}\rho
_{j}^{m_{j}}U_{m_{j}}(x)U_{m_{j}}(y_{j}),  \label{expk} \\
\prod_{j=1}^{k}(1-\rho _{j}^{2})/A_{2k}=\sum_{m_{1},m_{2},\ldots
m_{k}=0}^{\infty }V_{m_{1},\ldots ,m_{k}}\prod_{j=1}^{k}\rho
_{j}^{m_{j}}U_{m_{j}}(y_{j}),  \label{A2k}
\end{gather}%
where as before, above $V_{k_{1},\ldots ,k_{n}}\allowbreak =\allowbreak
\int_{-1}^{1}\frac{2}{\pi }\sqrt{1-x^{2}}\prod_{i=1}^{n}U_{k_{i}}(x)dx$.

Since each density $f_{Mk}$ can be presented as a linear combination of $%
f_{M1}(x|y_{i},\rho _{i})$ , $k\allowbreak =\allowbreak 1,\ldots ,k$ (by
simple fraction decomposition) we will analyze $f_{M1}$ first. We have the
following result:

\begin{theorem}
\label{K2}

i) $\forall y\in \lbrack -1,1]:$ $\int_{-1}^{1}f_{M1}(x|y,\rho
)dx\allowbreak =\allowbreak 1,$

ii) $\int_{-1}^{1}f_{M1}(x|y,\rho )\frac{2}{\pi }\sqrt{1-y^{2}}dy\allowbreak
=\allowbreak \frac{2}{\pi }\sqrt{1-x^{2}},$

iii) $\int_{-1}^{1}f_{M1}(x|y_{1},\rho _{1})f_{M1}(y_{1}|y_{2},\rho
_{2})dy_{1}\allowbreak =\allowbreak f_{M1}(x|y_{2},\rho _{1}\rho _{2}),$

iv) Polynomials orthogonal with respect to $f_{M1}$ are as follows: $%
P_{-1}(x|y,\rho )\allowbreak =\allowbreak 0$, $P_{0}(x|y,\rho )$, $%
P_{1}(x|y,\rho )\allowbreak =\allowbreak U_{1}(x)-2\rho y,$ and 
\begin{equation*}
P_{m}(x|y,\rho )=U_{m}(x)-2\rho yU_{m-1}(x)+\rho ^{2}U_{m-2}(x)
\end{equation*}%
for $m\geq 2.$
\end{theorem}

\begin{proof}
i) Either we use directly properties of $f_{K2}(x|a,b)$ and the fact that in
our case $ab\allowbreak =\allowbreak \rho ^{2},$ or we apply (\ref{P-M}).
ii) follows directly from i) and the fact that $w(x,y|\rho )\allowbreak
=\allowbreak w(y,x|\rho ).$ iii) we have: 
\begin{gather*}
\int_{-1}^{1}f_{M1}(x|y_{1},\rho _{1})f_{M1}(y_{1}|y_{2},\rho
_{2})dy_{1}\allowbreak = \\
\allowbreak \frac{2}{\pi }\sqrt{1-x^{2}}\int_{-1}^{1}\frac{2}{\pi }\frac{%
(1-\rho _{1}^{2})(1-\rho _{2}^{2})\sqrt{1-y_{1}^{2}}}{w(y_{1},x|\rho
_{1})w(y_{1},y_{2}|\rho _{2})}dy_{1}= \\
\frac{2}{\pi }(1-\rho _{1}^{2})(1-\rho _{2}^{2})\sqrt{1-x^{2}}%
/A_{4}\allowbreak =\allowbreak f_{M1}(x|y_{2},\rho _{1}\rho _{2}),
\end{gather*}%
since $A_{4}\allowbreak =\allowbreak \allowbreak (1-\rho _{1}^{2})(1-\rho
_{2}^{2})w(x,y_{2}|\rho _{1}\rho _{2})/(1-\rho _{1}^{2}\rho _{2}^{2}).$ iv)
We use assertions of Theorem \ref{p_ort} and Lemma \ref{szcz} i).
\end{proof}

\begin{remark}
Results of the Theorem \ref{K2} indicate possible applications of the
distributions $f_{M1}$ and Wigner in multivariate analysis and stochastic
processes. More precisely assertion i) shows that $f_{M1}(x|y,\rho )$ is in
fact a conditional distribution. ii) shows that $f_{M1}(x|y,\rho )f_{M0}(y)$
can be treated as a density of certain bivariate distribution with $f_{M0}$
that is Wigner distribution as its marginals. Finally iii) is nothing else
but the so-called Chapman--Kolmogorov property. These properties are known
and applied in stochastic processes, see, e.g. \cite{bms} and \cite%
{Szab-OU-W}. We quoted them for the sake of completeness of the paper and
also in order to present new proofs of these properties directly basing on
the general properties of generalized Kesten distributions discussed in the
first part of this paper.
\end{remark}

Following the above mentioned remark let us denote by $f_{2}(x,y|\rho )$ the
two-dimensional measure defined by:%
\begin{equation}
f_{2}(x,y|\rho )=f_{M1}(x|y,\rho )f_{M0}(y)=\frac{(1-\rho ^{2})\sqrt{%
(1-x^{2})(1-y^{2})}}{4\pi ^{2}w(x,y|\rho )}.  \label{2dim}
\end{equation}

\begin{remark}
Notice also that 
\begin{eqnarray*}
f_{M2}(x|y_{1},y_{2},\rho _{1},\rho _{2}) &=&\frac{2}{\pi }\frac{\sqrt{%
1-x^{2}}(1-\rho _{1}^{2})(1-\rho _{2}^{2})w(y_{1},y_{2}|\rho _{1}\rho _{2})}{%
w(x,y_{1}|\rho _{1})w(x,y_{2}|\rho _{2})(1-\rho _{1}^{2}\rho _{2}^{2})} \\
&=&\frac{f_{M1}(y_{1}|x,\rho _{1})f_{M1}(x|y_{2},\rho _{2})f_{M0}(y_{2})}{%
f_{M1}(y_{1}|y_{2},\rho _{1}\rho _{2})f_{M0}(y_{2})},
\end{eqnarray*}%
which can be interpreted in the following way. Let us consider 3 element
discrete Markov chain $X_{1},X_{2},X_{3}$ such that transition density $%
X_{2}|X_{3}$ is $f_{M1}(x|y_{2},\rho _{2}),$ transition $X_{1}|X_{2}$ is $%
f_{M1}(x|y_{1},\rho _{1})$ while marginal density of $X_{3}$ is $%
f_{M0}(y_{2})$ then the conditional density of $X_{2}|X_{1},X_{3}$ is $%
f_{M2}(x|y_{1},\rho _{1},y_{2},\rho _{2}).$
\end{remark}

\begin{lemma}
$\int_{-1}^{1}\frac{2}{\pi }\sqrt{1-y_{1}^{2}}\frac{w3(y_{1},y_{2},y_{3}|%
\rho _{1},\rho _{2},\rho _{3})}{w(y_{1},y_{2}|\rho _{1}\rho
_{2})w(y_{2},y_{3}|\rho _{2}\rho _{3})w(y_{1},y_{3}|\rho _{1}\rho _{3})}%
\allowbreak dy_{1}=\frac{1-\rho _{2}^{2}\rho _{3}^{2}}{w(y_{2},y_{3,}|\rho
_{2}\rho _{3})}.$
\end{lemma}

\begin{proof}
We start from (\ref{expk}) considered for $k\allowbreak =\allowbreak 3$ and
get 
\begin{gather*}
\prod_{j=1}^{3}(1-\rho _{i}^{2})/A_{6}\allowbreak =\allowbreak \frac{%
w3(y_{1},y_{2},y_{3}|\rho _{1},\rho _{2},\rho _{3})}{w(y_{1},y_{2}|\rho
_{1}\rho _{2})w(y_{2},y_{3}|\rho _{2}\rho _{3})w(y_{1},y_{3}|\rho _{1}\rho
_{3})}\allowbreak \\
=\sum_{m_{1}=0}^{\infty }\sum_{m_{2}=0}^{\infty }\sum_{m_{3}=0}^{\infty
}\rho _{1}^{m_{1}}\rho _{2}^{m_{2}}\rho
_{3}^{m_{3}}V_{m_{1},m_{2},m_{3}}U_{m_{1}}(y_{1})U_{m_{2}}(y_{2})U_{m_{3}}(y_{3})
\end{gather*}%
basing on Lemma \ref{szcz}. Now using (\ref{A2k}) and again Lemma \ref{szcz}
we get: 
\begin{gather*}
\int_{-}^{1}\frac{2}{\pi }\sqrt{1-y_{1}^{2}}\frac{w3(y_{1},y_{2},y_{3}|\rho
_{1},\rho _{2},\rho _{3})}{w(y_{1},y_{2}|\rho _{1}\rho
_{2})w(y_{2},y_{3}|\rho _{2}\rho _{3})w(y_{1},y_{3}|\rho _{1}\rho _{3})}%
dy_{1}\allowbreak \\
=\sum_{m_{2}=0}^{\infty }\sum_{m_{3}=0}^{\infty }\rho _{2}^{m_{2}}\rho
_{3}^{m_{3}}V_{0,m_{2},m_{3}}U_{m_{2}}(y_{2})U_{m_{3}}(y_{3})\allowbreak \\
=\allowbreak \sum_{m_{2}=0}^{\infty }(\rho _{2}\rho
_{3})^{m_{2}}U_{m_{2}}(y_{2})U_{m_{2}}(y_{3})\allowbreak =\allowbreak \frac{%
1-\rho _{2}^{2}\rho _{3}^{2}}{w(y_{2},y_{3,}|\rho _{2}\rho _{3})}.
\end{gather*}
\end{proof}

From this Lemma we derive the following important conclusion. Namely that
the following function: 
\begin{gather}
g(y_{1},y_{2},y_{3}|\rho _{1},\rho _{2},\rho _{3})=  \label{3Dden} \\
\frac{8}{\pi ^{3}}\sqrt{1-y_{1}^{2}}\sqrt{1-y_{2}^{2}}\sqrt{1-y_{3}^{2}}%
\frac{w3(y_{1},y_{2},y_{3}|\rho _{1},\rho _{2},\rho _{3})}{%
w(y_{1},y_{2}|\rho _{1}\rho _{2})w(y_{2},y_{3}|\rho _{2}\rho
_{3})w(y_{1},y_{3}|\rho _{1}\rho _{3})}  \notag
\end{gather}%
can be treated as the density of some $3D$ distribution with $2D$ marginals
equal to $f_{2M}(y_{1},y_{2}|\rho _{1}\rho _{2}),$ $f_{2M}(y_{1},y_{3}|\rho
_{1}\rho _{3}),$ $f_{2M}(y_{2},y_{3}|\rho _{2}\rho _{3}).$

\end{document}